\documentclass[a4paper,11pt]{amsart}
\usepackage{amssymb,mathrsfs,enumerate,graphicx, color}
\usepackage[pdfpagelabels,colorlinks,citecolor=red,urlcolor=blue,bookmarks=false,hypertexnames=true,backref]{hyperref} 
\usepackage[margin=2.8cm]{geometry}

\newtheorem{thm}{Theorem}[section]
\newtheorem{cor}[thm]{Corollary}
\newtheorem{lem}[thm]{Lemma}

\theoremstyle{definition}
\newtheorem{defn}[thm]{Definition}
\theoremstyle{remark}
\newtheorem{rem}[thm]{Remark}
 
 \numberwithin{equation}{section}

\newcommand{\opnorm}[1]{|\!|\!| #1 |\!|\!|}
\renewcommand{\Re}{\operatorname{Re}}
\newcommand{\lub}{\operatorname{l{.}u{.}b{.}}}

\newcommand{\C}{\mathbb{C}}
\newcommand{\N}{\mathbb{N}}
\newcommand{\R}{\mathbb{R}}
\newcommand{\Z}{\mathbb{Z}}

\newcommand{\expe}{\mathrm{e}}
\newcommand{\wrt}{\mathrm{d}}
\newcommand{\loc} {\mathrm{loc}}
\newcommand{\BMO}{\mathrm{BMO}} 

\title[On subordinated semigroups and applications]
{On subordinated semigroups and Hardy spaces associated to fractional powers of operators}         

\author{The Anh Bui}
\address{School of Mathematical and Physical Sciences, Macquarie University, Sydney, Australia}
\email{the.bui@mq.edu.au}

\author{Michael G. Cowling}
\address{    School of Mathematics and Statistics, University of New South Wales, Sydney, Australia}
\email{m.cowling@unsw.edu.au }

\author{Xuan Thinh Duong}
\address{School of Mathematical and Physical Sciences, Macquarie University, Sydney, Australia}
\email{xuan.duong@mq.edu.au}


\begin{document}    

\begin{abstract}  
Let $L$ be a positive self-adjoint operator on $L^2(X)$, where $X$ is a $\sigma$-finite metric 
measure space.
When $\alpha \in (0,1)$, the subordinated semigroup $\{\exp(-tL^{\alpha}):t \in \R^+\}$ can be defined on $L^2(X)$ and extended to $L^p(X)$. 
We prove various results about the semigroup $\{\exp(-tL^{\alpha}):t \in \R^+\}$, under different assumptions on $L$.
These include the weak type $(1,1)$ boundedness of the maximal operator $f \mapsto \sup _{t\in \R^+}\exp(-tL^{\alpha})f$ 
and characterisations of Hardy spaces associated to the operator $L$ by the area integral and vertical square function.
\end{abstract}
\date{}
\maketitle

\section{Introduction}
Let $L$ be a positive self-adjoint operator on $L^2(X)$ where $X$ is a metric measure space, with metric $\varrho$ and measure $\mu$. We may define the heat operators $H_z := \exp(-zL)$, where $z \in \C$ and $\arg(z) \leq \pi/2$, by spectral theory; these operators form an analytic semigroup.
Some of them extend continuously to $L^p(X)$ for some other values of $p$ (see \cite{CaDr}). 

When $0 < \alpha < 1$, the fractional powers $L^{\alpha}$ can be defined.
The study of  fractional powers $L^{\alpha}$, where $L$ is a Laplace operator, or a Schr\"odinger operator, or a Hardy operator, has attracted considerable interest recently: see, for example \cite{CS, CSV}.
These fractional powers of operators arise from partial differential equations that model real-world phenomena, as well as from applications in the mathematical sciences.
For instance, the fractional Laplacian on $\R^n$ is used to model a L\'evy jump process by superdiffusion in euclidean spaces, and the fractional Laplace--Beltrami operator is used for this process in curved manifolds.

The fractional powers $L^{\alpha}$ generate semigroups $\{ \exp(-tL^{\alpha}): t \in \R^+\}$ on $L^2(X)$.
It is well known that the operators $\exp(-tL^{\alpha})$ can be expressed in terms of the operators $\exp(-sL)$  by subordination, hence one can study $\exp(-tL^{\alpha})$ via the $\exp(-sL)$.
It is interesting that the operators $\exp(-tL^{\alpha})$ have some properties that the operators $\exp(-sL)$ need not possess. 
For example one can find upper bounds on the kernels of the time derivatives ${d^k \exp(-tL^{\alpha}) / dt^k}$, but not the kernels of the time derivatives ${d^k \exp(-tL) / dt^k}$, in terms of upper bounds for the kernels of $\exp(-tL)$. 
The aim of this paper is to study the subordinated semigroups $\{ \exp(-tL^{\alpha}): t \in \R^+\}$,  and some of their applications. 

This paper is structured as follows.
In Section 2, we take $\alpha$ to be $1/2$ and study the Poisson semigroup $\{ \exp(-tL^{1/2}): t \in \R^+\}$. 
In this case, the subordination formula is explicit and we can obtain sharp estimates for the Poisson operators $\exp(-tL^{1/2})$ and their time derivatives $L^{k/2} \exp(-tL^{1/2})$ when  
$t \in \R^+$, and for their analytic continuations $\exp(-zL^{1/2})$, when $z \in \C$ and $\left|\arg z \right| < {\pi / 4} $. 
Our results include  estimates of maximal operators associated to the Poisson semigroup on a sector in the case of Markov semigroups.

In Section 3, we recall the properties of L\'evy distributions, and then use these to study the
subordinated operators $\exp(-tL^{\alpha})$ for $t  \in \R^+$. 
We show that if $L$ is the generator of a symmetric Markov semigroup $\{ \exp(-tL): t \in \R^+\}$, then the maximal operator $f \mapsto \sup_{t \in \R^+} |\exp(-tL^{\alpha})f|$ is of weak type $(1,1)$. 
In \cite{Cowling}, the $L^p$ boundedness of this maximal operator was proved for all $p \in (1,\infty]$; this result covers the endpoint when $p=1$. 

In Section 4, we assume that the quasi-metric space $X$ is Ahlfors--David regular, that is, the volume of a ball satisfies $\mu (B(x,r)) \eqsim r^n$, for some $n$, and that $L$ is a positive self-adjoint operator on $L^2(X)$ for which the semigroup $\{ \exp(-tL): t \in \R^+\}$ on $L^2(X)$ has Poisson-type  upper bounds. 
We show that the Hardy space associated to the operator $L$ can be characterised by the area function and by the vertical square function of the subordinated semigroup $\{\exp(-tL^{\alpha}): t \in \R^+\}$ for $0 < \alpha < 1$. 
This new result proves that the Hardy space associated to $L$ is equivalent to the Hardy space associated to the fractional powers $L^{\alpha}$.

A note on notation: we write $\Gamma_{\beta}$ for the open sector $\{ z \in \C: \left| \arg z \right| < \beta  \}$ and $\bar\Gamma_{\beta}$ for its closure; we take the branch cut for $\arg$ (and hence also for complex powers) along the negative real axis.

\section{Poisson semigroups}\label{s1}

In this section, we study the Poisson semigroup, which is generated by $L^{1/2}$, that is, $P_t = \exp(-tL^{1/2})$ for all $t \in \R^+$. 
In this case, there is an explicit subordination formula that represents the Poisson semigroup in terms of the heat semigroup. 
We will employ this formula to continue $t \mapsto P_t$  analytically and obtain operators $P_z = \exp(-zL^{1/2})$ for $z \in \bar\Gamma_\beta$, where $\beta < \pi/4$, and estimate the time derivatives $P_{z,k} = (z L^{1/2})^k P_z = (-1)^k \, z^k \, {d^k P_z}/{dz^k} $ in terms of the operators $\exp(-tL^{1/2})$.

We first show that the time derivatives of the Poisson semigroup are bounded by the Poisson semigroup itself. 
Let $T$ be a bounded operator on $L^2(X)$. 
We say that $T$ has an associated kernel $k$ on $X\times X$ when $k$ is locally integrable on $X\times X$ with respect to the product measure $\mu\times \mu$, and moreover 
\[
\langle Tf,g\rangle =\int_X\int_X k(x,y)f(y)f(x)\,\wrt \mu(y)\,\wrt \mu(x)
\] 
for all $f,g\in L^2(X)$.

\begin{thm}\label{time-derivative}
Assume that
\begin{enumerate}[{(i)}]
\item $L$ is a positive self-adjoint operator on $L^2(X)$;

\item the heat semigroup $\{H_t= \exp(-tL): t  \in \R^+\}$ on $L^2(X)$ is positivity-preserving, in the sense that $\exp(-tL)f \ge 0$ if $f \geq 0$.
\end{enumerate} 
Then
\begin{enumerate}[{(a)}]
\item for each positive integer $k$ and $\theta \in (0,1)$, 
\[
|P_{t,k} f(x)| = | (t L^{1/2})^k \exp(-tL^{1/2}) f(x)|  
\lesssim_{k,\theta} P_{\theta t}(|f|)(x) 
\]
for all $x\in X$, all $t \in \R^+$ and all $f \in L^2(X)$.

\item Assume that the heat operators $H_t$ have kernels $h_t(x,y)$. 
Then the operators $P_t$ have kernels $p_t$ and the operators  $P_{t,k}$ have kernels $p_{t,k}$ and, for all positive integers $k$ and $\theta \in (0,1)$, 
\[
| p_{t,k} (x,y) |    \lesssim_{k,\theta}  p_{\theta t}(x,y)  \qquad \forall x,y\in X \quad\forall t \in \R^+.
\]
\end{enumerate}
\end{thm}

\begin{proof}
We start with the subordination formula: for all $\lambda \in \R^+$,
\[
\expe ^{-\lambda} 
= \frac{1 }{ \sqrt{\pi}} \int_0^{\infty} \expe ^{-s}  \expe ^{-\lambda ^2 / 4s} \,\frac{\wrt s}{s^{1/2}} ,
\]
so for a positive self adjoint operator $L$, 
\[
\exp(-tL^{1/2})f  =  \frac{1}{\sqrt{\pi}}\int_0^{\infty} \expe ^{-s}   \exp(-t^2L / 4s) f  \,\frac{\wrt s}{s^{1/2}} 
\qquad\forall f\in L^2(X) \quad\forall t \in \R^+.
\]

After the change of variable $u = t^2/4s$, we obtain
\begin{equation}\label{SubordinateFormula}
\exp(-tL^{1/2})f 
 =  \frac{1}{2\sqrt{\pi}}\int_0^\infty t \expe ^{-{t^2}/{4u}} \exp(-uL)f \,\frac{\wrt u}{u^{3/2}} \,.
\end{equation}
Hence
\begin{equation}\label{SubordinateFormula ptk}
\begin{aligned}
(tL^{1/2})^{k}\exp(-tL^{1/2})f&= (-1)^k\frac{t^k}{2\sqrt{\pi}}\int_0^\infty \partial^k_t(t\expe ^{-{t^2}/{4u}})\exp(-uL)f 
\,\frac{\wrt u}{u^{3/2}}\\
&=(-1)^{  k+1} \frac{1}{ \sqrt{\pi}}\int_0^\infty t^k \partial^{k+1}_t(\expe ^{-{t^2}/{4u}})\exp(-uL)f 
\,\frac{\wrt u}{u^{1/2}} .
\end{aligned}
\end{equation}
This, along with assumption (ii), yields the inequality
\begin{equation}\label{eq1-prop2.1}
\begin{aligned}
|P_{t,k}f|
&{ {}\le{}}\frac{1}{ \sqrt{\pi}}\int_0^\infty \frac{t^k}{u^{1/2}} \left|\partial^{k+1}_t(\expe ^{-{t^2}/{4u}})\right| \exp(-uL)(|f|) \,\wrt u.
\end{aligned}
\end{equation}

Let $u\in \R^+$ and $k\in \N$. 
By Fa\`a di Bruno's formula, we can write
\[
\partial^{k+1}_t (\expe ^{-{t^2}/{4u}} )
=\sum \frac{(-1)^{m_1+m_2}}{ (2m_1)! \, m_2!} \expe ^{-{t^2}/{4u} }\Big(\frac{t}{4u}\Big)^{m_1}\Big(\frac{1}{4u}\Big)^{m_2},
\]
where the sum is taken over all pairs $(m_1,m_2)$ of nonnegative integers satisfying $m_1+2m_2=k+1$. 
For every  such pair $(m_1,m_2)$, and any positive $\epsilon$,
\begin{align*}
\frac{t^k}{u^{1/2}} \Big(\frac{t}{4u}\Big)^{m_1}\Big(\frac{1}{4u}\Big)^{m_2}
=  \frac{1}{4^{m_1+m_2}}\frac{t}{u^{3/2}}  \frac{t^{2m_1 + 2m_2 -2} } {u^{m_1+m_2 - 1} } 
\leq \frac{t}{u^{3/2}} \frac{(m_1 + m_2 -1)!}{4 (4\epsilon)^{m_1+m_2 - 1}} \expe ^{{\epsilon t^2}/{4u}}  .
\end{align*}
This implies that for all $\theta \in (0,1)$ there exists $c(\theta ,k)$ such that
\begin{equation}\label{eq2-prop2.1}
\frac{t^k}{u^{1/2}} \left| \partial^{k+1}_t \expe ^{-{t^2}/{4u}}\right|
\leq c(\theta ,k)  \frac{t}{u^{3/2}} \expe ^{-{( \theta t)^2}/{4u}}.
\end{equation}

From \eqref{eq1-prop2.1} and \eqref{eq2-prop2.1}, we deduce that
\[
\begin{aligned}
|P_{t,k}f|
&\leq \frac{c(\theta,k)}{\sqrt{\pi}}\int_0^\infty t\expe ^{-{(\theta t)^2}/{4u}}\exp(-uL)(|f|) \,
\frac{\wrt u}{u^{3/2}}\\
&=\frac{c(\theta,k)}{\theta \sqrt{\pi}}\int_0^\infty (\theta t)\expe ^{-{(\theta t)^2}/{4u}}\exp(-uL)(|f|) \,\frac{\wrt u}{u^{3/2}}\\
&=C(\theta,k) P_{\theta t}(|f|),
\end{aligned}
\]
say, where in the last equality we used \eqref{SubordinateFormula}.
This completes the proof of (a).

We now prove (b).
For all $f, g \in L^2(X)$,
\[
\langle \exp(-tL)f, g \rangle = \int_X \int_X h_t(x,y) f(x) g(y) \,\wrt \mu(x) \,\wrt \mu(y) .
\]
Since $L$ is positive and self-adjoint, $| \langle \exp(-tL)f, g \rangle | \le \| f \|_2 \| g \|_2$.
From the subordination formula, we see that
\[
\bigl\langle \exp(-tL^{1/2})f, g \bigr\rangle 
= \frac{1}{ 2\sqrt{\pi}} 
\int_X \int_X \int_0^{\infty} t\expe ^{{-t^2}/{4u}} h_u(x,y) f(x) g(y) \,\frac{\wrt u}{u^{3/2}} \,\wrt \mu(x) \,\wrt \mu(y) .
\]
We change the order of integration, integrating first  with respect to $x$ and $y$, and obtain
\[
\bigl|\bigl\langle \exp(-tL^{1/2})f, g \bigr\rangle\bigr| 
\leq \frac{1}{ 2\sqrt{\pi}} \| f \|_2 \| g \|_2 \int_0^{\infty} t\expe ^{{-t^2}/{4u}}\,\frac{\wrt u}{u^{3/2}} .
\]
A further integration now gives
\[
\bigl|\bigl\langle \exp(-tL^{1/2})f, g \bigr\rangle\bigr| \le  \| f \|_2 \| g \|_2 .
\]

We now show that the kernel
\[
(x,y) \mapsto \int_0^{\infty}  t\expe ^{{-t^2}/{4u}}  h_u(x,y) \frac{\wrt u }{ u^{3/2}}
\]
is locally integrable in $X\times X$.
If $A$ is a bounded set in $X\times X$, then $A\subset \Omega\times \Omega$ for some bounded set  $\Omega$ in $X$, and
\[
\begin{aligned}
&\iint_A \biggl|\int_0^{\infty} t\expe ^{{-t^2}/{4u}} h_u(x,y) \,\frac{\wrt u}{ u^{3/2}}\biggr| \,\wrt \mu(x) \,\wrt \mu(y) \\
&\qquad=\iint_A \int_0^{\infty} t\expe ^{{-t^2}/{4u}} h_u(x,y) \,\frac{\wrt u}{u^{3/2}} \,\wrt \mu(x) \,\wrt \mu(y)\\
&\qquad\le\int_X\int_X \int_0^{\infty}t\expe ^{{-t^2}/{4u}} h_u(x,y) 1_\Omega(y)1_\Omega(x) \,\frac{\wrt u}{u^{3/2}} \, d\mu(x) \,\wrt \mu(y).
\end{aligned}
\]
From above, 
\[
\begin{aligned}
\iint_A \biggl| \int_0^{\infty} t\expe ^{{-t^2}/{4u}} h_u(x,y) \,\frac{\wrt u}{u^{3/2}} \biggr| \,\wrt \mu(x) \,\wrt \mu(y)
\le \|1_\Omega\|_2\|1_\Omega\|_2 \,,
\end{aligned}
\]
which shows that 
\[
(x,y) \mapsto \int_0^{\infty} t\expe ^{{-t^2}/{4u}}  h_u(x,y) \,\frac{\wrt u}{ u^{3/2}}
\]
is locally integrable.

Thus the kernel of $\exp(-tL^{1/2})$ is given by
\[
p_t(x,y) 
=  \frac{1}{2\sqrt{\pi}} \int_0^{\infty} t\expe ^{{-t^2}/{4u}} h_u(x,y) \,\frac{\wrt u}{u^{3/2}} \,,
\]
that is, 
\[
\bigl\langle\exp(-tL^{1/2})f, g \bigr\rangle 
= \int_X \int_X p_t(x,y) f(x) g(y) \,\wrt \mu(x) \,\wrt \mu(y) 
\qquad\forall f, g \in L^2(X).
\]
Now, as in the proof of \eqref{SubordinateFormula ptk}, we obtain the kernels of the derivatives:
\[
p_{t,k}(x,y)
=(-1)^{  k+1} \frac{t^k}{ \sqrt{\pi}}\int_0^\infty\partial^{k+1}_t(\expe ^{-{t^2}/{4u}})h_u(x,y) \,\frac{\wrt u}{u^{1/2}} \,.
\]
From the assumption (i), we see that $p_t(x,y)\geq 0$ for all $x,y\in X$ and $t \in \R^+$. 
Finally, arguing as in the proof of  (a), we obtain (b).
\end{proof}

We now give similar estimates for complex Poisson semigroups.

\begin{thm}\label{complex-semigroup}
Assume that $L$ satisfies the assumptions (i) and (ii) of Theorem \ref{time-derivative}. 
Fix $\beta \in (0,\pi/4)$, and set $\gamma = (1-\tan^2\beta)^{1/2}$.
Then
\begin{enumerate}[{(a)}]
\item 
the Poisson semigroup $\{ P_t : t\in \R^+\}$ can be extended to a complex semigroup
$\{ P_z : z \in \bar\Gamma_\beta\}$. 
Moreover, if $z \in \bar\Gamma_\beta$ and $\tau \leq \Re z \leq 4 \tau$, then 
\[
| P_z f(x) |   \le 4\sqrt{2}\gamma^{-1} P_{\gamma \tau} (|f|)(x)
\qquad\forall x\in X .
\]

\item 
Assume that the heat operators $H_t$ have kernels $h_t(x,y)$. 
Then the Poisson operators $P_z$ have kernels $p_z$, and if $z \in \bar\Gamma_\beta$ and $\tau \leq \Re z \leq 4\tau$, then 
\[
| p_z (x,y) |    \le 4\sqrt{2}\gamma^{-1} p_{\gamma \tau}(x,y)
\qquad\forall x,y\in X .
\]
\end{enumerate}
\end{thm}

\begin{rem}
The number $4$ may be replaced by any other number greater than $1$, at the cost of changing the constant $4\sqrt{2}\gamma^{-1}$.
\end{rem}

\begin{proof}
Recall the subordination formula \eqref{SubordinateFormula}:
\[
\exp(-tL^{1/2})f(x)  =  \frac{1}{2\sqrt{\pi}}\int_0^\infty t\expe ^{-{t^2}/{4u}} \exp(-uL)f(x) \,\frac{\wrt u}{u^{3/2}}
\qquad\forall f\in L^2(X).
\]
The assumption (i) implies that the heat and Poisson semigroups extend analytically to $\Gamma_{\pi/2}$ as semigroups on $L^2(X)$.
Fix $\tau \in \R^+$ and $z \in \bar\Gamma_{\beta}$ such that $\tau \leq \Re z \leq 4 \tau$. 
By analytic continuation,
\[
\exp(-zL^{1/2})f(x) = \frac{1}{2\sqrt{\pi}}\int_0^\infty z\expe ^{-{z^2}/{4u}} \exp(-uL)f(x) \,\frac{\wrt u}{u^{3/2}  }  \,.
\]
This, along with the assumption (ii), yields
\begin{equation}\label{eq-pzxy}
|\exp(-zL^{1/2})f(x) | 
= \frac{1}{2\sqrt{\pi}}\int_0^\infty \left|z\expe ^{-{z^2}/{4u}}\right| \exp(-uL)(|f|)(x) \,\frac{\wrt u}{u^{3/2}}  \,.
\end{equation}
Write $z = t+is$, where $t,s \in \R$.
Then $|s| \leq t \tan \beta < t$ since $\beta< {\pi}/{4}$.
Therefore 
\[
|z|^2 = |t|^2 + |s|^2 \leq 32 \tau^2
\qquad\text{and}\qquad
\Re(z^2) = t^2 - s^2 \geq (1 - \tan^2\beta) t^2 \geq (\gamma \tau)^2.
\]

Hence
\[
\left| z \expe ^{-{z^2}/{4u}}\right|\leq \frac{4\sqrt{2}}{\gamma}  (\gamma \tau) \expe ^{-{(\gamma \tau)^2}/{4u}}.
\]
Inserting this into \eqref{eq-pzxy}, we obtain
\[
\begin{aligned}
|\exp(-zL^{1/2})f(x) | 
&\leq \frac{4\sqrt{2}}{\gamma} \frac{1}{2\sqrt{\pi}}\int_0^\infty (\gamma \tau) \expe ^{-{(\gamma \tau)^2}/{4u}} \exp(-uL)(|f|)(x) \,\frac{\wrt u}{u^{3/2}} \\
&=\frac{4\sqrt{2}}{\gamma}  P_{\gamma \tau} (|f|)(x),
\end{aligned}
\]
where we used \eqref{SubordinateFormula} in the last inequality, which proves (a).

The proof of (b) is similar and we omit the details.
\end{proof}

The argument of Theorem \ref{complex-semigroup} also proves Theorem \ref{time-derivative}, because the derivatives of an analytic function can be expressed via the Cauchy formula in terms of the function.

\begin{cor} 
Theorem \ref{complex-semigroup} implies Theorem \ref{time-derivative}.
\end{cor}

\begin{proof} 
Take $\theta \in (0, 1)$.
To obtain the estimate of Theorem \ref{time-derivative}(a), namely
\[
| P_{t,k} f(x)| 
= | (t L^{1/2})^k \exp(-tL^{1/2}) f(x)| 
\le C_{\theta} P_{\theta t} (| f |)(x) 
\qquad\forall t \in \R^+,
\]
we choose $\beta \in (0,\pi/4)$ such that $(1 -\tan^2\beta)^{1/2} (1-\sin\beta) = \theta$, we let $\gamma$ be $(1-\tan^2\beta)^{1/2}$ and $\tau$ be $(1 - \sin\beta)t$, and we define $C$ to be the circle centred at $t$ with radius $r = t \sin \beta$.

From the Cauchy integral formula, 
\[
P_{t,k} f 
= (-t)^{k} \frac{\wrt ^k }{ dt^k} \exp(-tL^{1/2})f = \frac{(-t)^k k! }{ 2\pi i}  \int_C \frac{\exp(-zL^{1/2})f }{ (z-t)^{k+1} } \,\wrt z .
\]
Every point on the circle $C$ has the property that $\tau \leq \Re z \leq 4\tau$, so, from Theorem \ref{complex-semigroup},
\[
\left| P_{t,k} f (x) \right|
\leq  \frac{t^k k! }{ 2\pi }  \int_C \frac{| \exp(-zL^{1/2})f(x) | }{ r^{k+1} } \,\wrt z 
\leq \frac{t^k k! }{ r^{k} }  P_{\gamma \tau}(|f|)(x)   
= \frac{4 \sqrt{2} \,k!} {(\sin\beta)^{k} \gamma} P_{\theta t}(|f|)(x)   .
\]

This proves part (a) of Theorem \ref{time-derivative}.
The proof of part (b) is similar.
\end{proof} 

We can replace assumption (ii) by a weaker assumption: instead of assuming that the heat semigroup is positivity-preserving, we can assume that the heat semigroup is dominated by a positivity-preserving semigroup, and then obtain similar results to Theorems \ref{time-derivative} and \ref{complex-semigroup}. 
We state this in the next theorem.

Given a semigroup $\{G_t: t \in \R^+\}$ with generator $M$, we write $\{Q_t : t \in \R^+\}$ for the associated Poisson semigroup, that is, $Q_t = \exp(-tM^{1/2})$, and write $q_t$ for the associated kernel, if it exists.

\begin{thm}\label{domination-semigroup}
Assume that
\begin{enumerate}[{(i)}]
\item $L$ is a positive self-adjoint operator on $L^2(X)$;

\item the heat operators $H_t$ are dominated by the operators $G_t$ of a positivity-preserving semigroup $\{\exp(-tM): t \in \R^+\}$, in the sense that  $| H_t f  |  \le C G_t |f| $ for all $f \in L^2(X)$ and all $t \in \R^+$.
\end{enumerate} 
Fix $\beta \in (0,\pi/4)$.
Then
\begin{enumerate}[{(a)}]
\item for each positive integer $k$ and $\theta \in (0,1)$, for all $f\in L^2(X)$,
\[
|P_{t,k} f(x)| = | (t L^{1/2})^k \exp(-tL^{1/2}) f(x)|  \lesssim_{k,\theta} {Q}_{\theta t}(|f|)(x) 
\qquad \forall x\in X.
\]

\item the Poisson semigroup $\{ P_t : t\in \R^+\}$ can be extended to a complex semigroup
$\{ P_z : z \in \bar\Gamma_\beta\}$. 
Moreover there exists $\gamma=\gamma(\beta) \in (0,1)$ such that
\[
| P_z f(x) | = | \exp(-z L^{1/2}) f(x) | \le \sqrt{2}\gamma^{-1} {Q}_{\gamma t} (|f|)(x)
\qquad \forall x\in X ,
\]
where $t = \Re z$.
\end{enumerate}
Assume further that the heat operators $H_t$ and the dominating operators $G_t$  have kernels $h_t$ and $g_t$ in $L^1_{\loc}(X\times X)$. 

\begin{enumerate}[(a)]\setcounter{enumi}{2}
\item Then the operators $Q_t$, $P_t$ and $P_{t,k}$  have kernels $q_t$,  $p_t$ and  $p_{t,k}$  in $L^1_{\loc}(X\times X)$, and further, for each positive integer $k$ and $\theta \in (0,1)$, 
\[
| p_{t,k} (x,y) |    \lesssim_{k,\theta} {q}_{\theta t}(x,y)
\qquad\forall x,y\in X.
\]

\item 
For $\left| \arg z\right| \le \beta$ with any fixed $\beta < {\pi / 4}$, there exists $\gamma=\gamma(\beta) \in (0,1)$ such that
\[
| p_z (x,y) |    \le \sqrt{2}\gamma^{-1} g_{\gamma t}(x,y)
\qquad\forall x,y\in X ,
\]
where $t = \Re z$ and $p_z(x,y)$ is the kernel of $P_z$.
\end{enumerate}
\end{thm}

\begin{proof}
The proof is similar to those above. 
We will only show the modification for the proof of (a), since the modifications for parts (b), (c), (d) are similar.

As in \eqref{SubordinateFormula ptk},
\begin{equation}\label{SubordinateFormula ptk2}
(tL^{1/2})^{k}\exp(-tL^{1/2})f 
= (-1)^k\frac{t^k}{4\sqrt{\pi}}\int_0^\infty\partial^{k+1}_t(\expe ^{-{t^2}/{4u}})\exp(-uL)f \,\frac{\wrt u}{\sqrt{u}}
\end{equation}
which implies that
\[
| (t L^{1/2})^{k} \exp(-tL^{1/2} ) f|  
\le   \frac{C t^k}{4\sqrt{\pi}}\int_0^\infty \left| \partial^{k+1}_t(\expe ^{-{t^2}/{4u}}) \right|  \exp(-uM) |f| \,\frac{\wrt u}{\sqrt{u}} .
\]
The rest of the proof proceeds as before, but with $\exp(-tM) |f| $ in place of $\exp(-tL) | f |$.
\end{proof}



We now consider the case where $\{H_t:t \in \R^+\}$ is a {symmetric} Markov semigroup, that is, $L$ is positive self-adjoint and $\{H_t:t \in \R^+\}$ extends continuously to a positivity-preserving contraction semigroup on $L^p(X)$ for all $p$ in $[1, \infty]$. 
It then follows from the subordination formula that the Poisson semigroup $\{P_t : t \in \R^+\}$ is also a Markov semigroup. 

Let $P^{*} f = \sup_{t > 0} | P_t f | $ for all $f \in L^p(X)$. 
The operator $P^{*}$ is bounded on $L^p(X)$ when $1 < p \le \infty$ ({{see \cite{Stein, Cowling}}}). 
The following result is a direct consequence of Theorems \ref{time-derivative} and \ref{complex-semigroup}.

\begin{thm}\label{maximal-estimate}
Assume that $\{H_t:t \in \R^+\}$ is a symmetric Markov semigroup. 
For each nonnegative integer $k$, set
\[
M^{*}_k f  = \sup_{z \in  \bar\Gamma_{\beta}}  \left| z^k \frac{\wrt ^k }{ dz^k} P_z  f \right|  .
\] 
Then the operators $M^{*}_k$ are bounded on $L^p(X)$ for all $p$ in $(1, \infty]$. 
\end{thm}
\begin{proof}
Theorem \ref{maximal-estimate} follows from the pointwise estimates in Theorem \ref{time-derivative}(a) and 
Theorem \ref{complex-semigroup}(a) together with the $L^p$ boundedness of $P^{*}$.
\end{proof}
\begin{rem}
A standard example of a symmetric Markov semigroup is $\{\exp(-t\Delta):t \in \R^+\}$, where $\Delta$ is (minus) the Laplace--Beltrami operator on a complete Riemannian manifold $X$.
\end{rem}




\section{ L\'evy distributions and subordinated semigroups}

\subsection{L\'evy distributions}
We shall now study the subordinated semigroup $\{\exp(-tL^{\alpha}): t \in \R^+\}$ when $\alpha \in(0, 1)$.
While the Poisson semigroup in the special case when $\alpha = 1/2$ has an explicit subordination formula, the general case  leads us to  the strictly stable L\'evy distribution with stability parameter $\alpha$. For information about these see, for example,  Sato \cite[\S14]{Sato} or Zolotarev \cite{Zolotarev}.
In what follows, many objects depend on $\alpha$, but we do not indicate this explicitly.

A Laplace transform and contour integration argument shows that
\[
\expe ^{-u s^\alpha} = \int_0^\infty p_u(t) \, \expe ^{-ts} \,\wrt t
\qquad\forall s, u  \in \R^+
\]
if and only if
\[
\begin{aligned}
p_u(t)
&= \frac{1}{2\pi i} \int_{0-i\infty}^{0+i\infty} \expe ^{-u z^\alpha + tz} \,\wrt z  \\
&= \frac{1}{2\pi } \int_{-\infty}^{+\infty} \expe ^{-u (iy)^\alpha + ity} \,\wrt y  \\
&= \mathcal{F}^{-1} m_u(t) 
\qquad\forall t \in \R^+,
\end{aligned}
\]
where $\mathcal{F}$ denotes the Fourier transformation and
\[
m_u(y) = \expe ^{-u (iy)^\alpha} = \expe ^{-u[\cos(\pi\alpha/2) + i \sin(\pi\alpha/2)\operatorname{sgn}(y)] \, |y|^\alpha}
\qquad\forall y \in  \R.
\]
Since $m_u$ extends continuously and boundedly to the lower half-plane $\C_-$ of complex numbers with nonpositive real part, $ \mathcal{F}^{-1} m_u(t) $ vanishes when $t<0$.
It is therefore consistent to suppose that $p_u$ is defined on $\R$ and vanishes on $\R^-$.
Once we do this, we may view $m_u$ as $\mathcal{F} p_u$, or its analytic extension to the half-plane as a version of the Laplace transform of $p_u$ that has been rotated in the complex plane.

The function $m_u$ vanishes at infinity faster than any polynomial, and is infinitely differentiable away from $0$; its derivatives also vanish faster than any polynomial at infinity.
However, it is not smooth at $0$.
Indeed, take a compactly supported bump function $\phi$ that takes the value $1$ near $0$, and write
\[
m_u(y) = (1 - \phi(y)) m_u(y) + \phi(y) + \sum_{n=1}^\infty \phi(y) \,  \frac{(-u)^n}{n!} \,\omega(y)^n \, |y|^{n\alpha},
\]
where $\omega(y) = [\cos(\pi\alpha/2) + i \sin(\pi\alpha/2)\operatorname{sgn}(y)]$.
The first two terms are smooth, and the least smooth of the other terms is when $n = 1$.
This implies that $p_u$ is smooth but $p_u(t) \eqsim t^{-1-\alpha}$ when $t \to \infty$.
The derivatives of $p_u$ behave better:   $p'_u(t) \eqsim t^{-2-\alpha}$, $p''_u(t) \eqsim t^{-3-\alpha}$, and so on when $t \to \infty$.
We summarise this in a family of inequalities: 
\begin{equation}\label{eq:p-estimates}
\begin{aligned}
| p_1(t) | \lesssim (1+t) ^{-1-\alpha} \qquad\forall t \in \R^+ \\
| t \, p'_1(t) | \lesssim (1+t) ^{-1-\alpha} \qquad\forall t \in \R^+  \\
| t^2 \, p''_1(t) | \lesssim (1+t) ^{-1-\alpha} \qquad\forall t \in \R^+  \\
\end{aligned}
\end{equation}
and so on; here the implicit constants depend only on the order of the derivative and $\alpha$.

It is important for us that $p_u$ is the probability density function for the L\'evy stable distribution, and hence $p_u$ is nonnegative and has total mass $1$.
Finally, the different $p_u$ form a semigroup and are related to each other by a change of variable:
\[
p_u * p_v = p_{u+v}
\qquad\text{and}\qquad
p_u(t) = u^{-1/\alpha} \, p_1(u^{-1/\alpha}t )
\]
for $u, v \in \R^+$.

We now discuss analytic continuation; to facilitate this, we write $H^\infty(\C^-)$ for the Banach space of bounded holomorphic functions on the lower complex half plane (with the topology of uniform convergence on compacta in $\C^-$).
The functions $m_u$ and $p_s$ admit analytic continuations in several different ways.
Indeed,
\[
m_u(y) = \expe ^{-u (iy)^\alpha}
\qquad\forall y \in \R,
\]
and this function extends to the function $w \mapsto \expe ^{-u (iw)^\alpha}$ in $H^\infty(\C^-)$, which implies that the tempered distribution $\mathcal{F}^{-1}m_u$ is supported in $[0, +\infty)$, by the Paley--Wiener theorem.
But it is also possible to continue $m_u(y)$ analytically in $u$; indeed, for $w$ in the lower half plane,
\[
|\arg((iw)^\alpha)| < \frac{\pi\alpha}{2} \,,
\]
so that when moreover $|\arg(u)| < \pi(1-\alpha)/2$,
\[
|\arg( u (iw)^\alpha )| < \frac{\pi}{2}
\]
and $w \mapsto \expe ^{-u (iw)^\alpha}$ is still bounded and analytic in the lower half plane.
Thus $u \mapsto m_u$ may be extended to a holomorphic map of the cone $\Gamma_{ \pi(1-\alpha)/2}$ into the space $H^\infty(\C^-)$.

By using the inverse Fourier transformation and the Paley--Wiener theorem, we deduce that $u \mapsto p_u$ may be extended to a holomorphic map of the cone $\Gamma_{ \pi(1-\alpha)/2}$ to the space of tempered distributions on $\R$ which are supported in $[0, +\infty)$.
Then, by repeating the analysis described above of decay at infinity and the discontinuity at $0$ of $m_u$, considered as a function of $\R$, we may conclude that each $p_u$ is in fact given by integration against a smooth function which is integrable at infinity; finally, we may observe that $p_u(t)$ may be viewed as an analytic function in both variables $t$ and $u$.
Of course, the analytically continued functions need not be nonnegative.



\subsection{Spectral theory}\label{subsec:spectral-theory}
Let $X$ be a set, equipped with a positive measure $\mu$; suppose that $X$ is $\sigma$-finite.
Denote the corresponding Lebesgue space by $L^p(X)$.
If $L$ is a positive operator on $L^2(X)$, such that the operator $\exp(-tL)$, defined initially on $L^2(X)$ by spectral theory, extends continuously to a contraction operator on $L^p(X)$. for all $p \in [1, \infty)$, then we may also write
\[
\exp(-u s^\alpha L^\alpha) = \int_0^\infty p_u(t) \, \exp(-tsL) \,\wrt t
\qquad\forall s \in \R^+
\]
by spectral theory, and then
\[
\opnorm{ \exp(-u s^\alpha L^\alpha)  }_p
\leq \int_0^\infty p_u(t) \, \opnorm{ \exp(-tsL)  }_p \,\wrt t
\leq \int_0^\infty p_u(t) \,\wrt t
\leq 1,
\]
where $\opnorm{ \cdot  }_p$ denotes the operator norm on $L^p(X)$.
Hence the operator $\exp(s^\alpha L^\alpha)$ also extends continuously from $L^2(X)$ to $L^p(X)$ for all $p \in [1, \infty)$.

The operators $\exp(-tL)$ form a symmetric  Markov  semigroup, of the sort that was treated by Stein \cite{Stein} and Cowling \cite{Cowling}.
In particular, these authors considered averaging operators $A^s$, defined by
\[
A^s =  \frac{1}{s} \, \int_0^{s} \exp(-tL) \,\wrt t,
\]
which behave better than the original operator $\exp(-tL)$, and may be treated using the Hopf--Dunford--Schwartz  maximal ergodic theorem (see \cite[Section VIII.6.4]{DS}).
More precisely, for $f \in L^p(X)$, where $1 \leq p \leq \infty$, we write $A^*f$ for the maximal function $\sup_{s \in\R^+} | A^s f |$. 

\begin{thm}
Let $A^*$ be as just defined.
If $p=1$, then
\[
\mu \{ x \in X : A^*f(x) > \lambda \} \leq 2 \,\frac{\Vert f \Vert_1}{\lambda}
\qquad\forall \lambda \in \R^+ \quad\forall f \in L^1(X) ,
\]
while if $1 < p \leq \infty$, then
\[
\| A^*f \|_p \leq 2 \left( \frac{p}{p-1}\right)^{1/p} \, \| f \|_p  
\qquad \forall f \in L^p(X).
\] 
\end{thm}

\subsection{The maximal function }
We have already and will continue to consider the suprema of uncountable familes of measurable functions (modulo identification almost everywhere), and at first sight, these might lead to serious problems with measurability. 
Fortunately, some results of Dunford and Schwartz \cite[IV.12]{DS} resolve this issue.
To state these with appropriate care, we need some definitions.

We consider the space $\mathcal{M}(X)$ of measurable real-valued functions, each of which is defined on a subset of $X$ of full measure.
We say that two such functions $f$ and $g$ are equivalent when the set $\{ x \in X : f(x) = g(x) \}$ is of full measure in $X$, and write $M(X)$ for the space of equivalence classes of functions in $\mathcal{M}(X)$ modulo this equivalence; temporarily we denote the equivalence class of $f$ in $\mathcal{M}(X)$ by $\dot f$.
Thus $M(X)$ may be ordered: $\dot f \leq \dot g$ if and only if the set $\{ x \in X : f(x) \leq g(x) \}$ is of full measure in $X$ for one, and hence every, pair of representatives $f$ and $g$ of the equivalence classes $\dot f$ and $\dot g$.
Now we may define upper bounds and least upper bounds for sets of equivalence classes in the obvious way; we write $\lub$ for a least upper bound, if it exists.

\begin{lem}
Let $X$ be a set, equipped with a positive measure $\mu$; suppose that $X$ is $\sigma$-finite.
Let $\mathcal{F}$ be a family of functions $f$ in $\mathcal{M}(X)$, and $\dot{\mathcal{F}}$ denote the corresponding family of equivalence classes $\dot f$ in $M(X)$.
Suppose that $\dot{\mathcal{F}}$ is bounded above in $M(X)$ by $\dot g$.
Then $\lub \dot{\mathcal{F}}$ exists in $M(X)$, and is the equivalence class of $\sup\{ f : f \in \mathcal{F}_0 \} $, for a suitable countable subfamily $\mathcal{F}_0$ of $\mathcal{F}$.
If $1 \leq p \leq \infty$ and $g \in L^p(X)$, then $\lub \dot{\mathcal{F}} \in L^p(X)$ too.
The corresponding result also holds if $\dot g$ is in a Lorentz space on $X$. 
\end{lem}

In light of the lemma, we may write $\sup\{f: f \in \dot{\mathcal{F}}\}$ in place of $\lub \dot{\mathcal{F}}$.
We now revert to the standard practice of not distinguishing between functions and equivalence classes of functions.

Now we generalise Stein \cite[Section III.3]{Stein}.
Suppose that $X$, $\mu$, and $L$ are as in Subsection \ref{subsec:spectral-theory}.
For a continuously differentiable function $h : \R \to \C$ that vanishes on $(-\infty, 0]$, we write $\tilde h$ its Laplace transform, and define the corresponding operator $\tilde h(L)$ as follows:
\[
\tilde h(L) = \int_0^\infty h(t) \, \exp(-tL) \,\wrt t.
\]

We now give the main result of this subsection.
\begin{thm}
Suppose that $C$ is a constant and the continuously differentiable function $h: [0,+\infty) \to \C$ satisfies the conditions
\[
\lim_{t \to 0+} |t\,h(t)|  = \lim_{t \to +\infty} |t\,h(t)| = 0
\]
and
\[
\int_0^\infty \,| t\,h'(t)| \,\wrt t \leq C.
\]
Then for all $f \in L^p(X)$,  
\[
| \tilde h(L) f | \leq C \,A^* f.
\]
Consequently , if $\mathcal{H}_C$ denotes the family of functions $h$ satisfying these conditions, then for all $f \in L^p(X)$,
\[
\sup\{  |\tilde h(L)f| : h \in \mathcal{H} \} \leq C \, A^*f.
\]
\end{thm}

\begin{proof}
Observe that
\[
\begin{aligned}
\tilde h(L)
&= \int_0^\infty h(t) \, \exp(-tL) \,\wrt t \\
&= \int_0^\infty h(t) \, \frac{\wrt }{dt} \left(\int_0^t \exp(-sL) \,\wrt s \right) \,\wrt t\\
&= -\int_0^\infty h'(t) \int_0^t \exp(-sL) \,\wrt s  \,\wrt t\\
&= -\int_0^\infty t\,h'(t) \left( \frac{1}{t} \, \int_0^t \exp(-sL) \,\wrt s \right)  \,\wrt t\\
&= -\int_0^\infty t\,h'(t) \, A^{t} \,\wrt t.
\end{aligned}
\]
It follows that, for $f \in L^p(X)$,
\[
\begin{aligned}
|\tilde h(L) f|
&\leq \left| \int_0^\infty t\,h'(t) \,A^{t}f \,\wrt t \right| \\
&\leq \left( \int_0^\infty |t\, h'(t)| \,\wrt t \right) \, A^* f \\
&\leq C A^* f.
\end{aligned}
\]
The result now follows from the previous lemma.
\end{proof}

We want to apply this result to families of functions $h$, and want the estimates to be dilation invariant, so that we get the same estimate for all $\exp(-t L^\alpha)$ as $t$ varies onver $\R^+$.
The next lemma does this.
Recall that $\tilde h$ is the Laplace transform of $h$, and note that the Laplace transform of the function $t \mapsto t \, h'(t)$ is the function $\lambda \mapsto \tilde h(\lambda) + \lambda \, \tilde h'(\lambda)$.

\begin{lem}
The inequality 
\[
\begin{aligned}
\int_0 ^{+\infty} |t \, h'(t)| \, dt 
&\lesssim \sup_s |\tilde h(is) | +  \sup_s |s \tilde h'(is) | +  \sup_s |s^2 \tilde h''(is) | +  \sup_s |s^3 \tilde h'''(is) |  \\
&\qquad + \biggl( \sup_s |s^2 \tilde h(is)  | \sup_s |\tilde h ''(is)| + \sup_s |s^2 \tilde h(is) | \sup_s |s \tilde h '''(is)| \\
&\qquad\qquad\qquad + \sup_s |s^3 \tilde h'(is) | \sup_s |\tilde h ''(is)|  \biggr)^{1/2}  
\end{aligned}
\]
holds for all continuously differentiable functions $h: \R \to \C$ supported in $[0, + \infty)$, whose Laplace transform $\tilde h$ is analytic on $\C^-$ and continuously differentiable and bounded on the closure of $\C^-$.
\end{lem}

\begin{proof}
Suppose that $g : \R \to \C$ satisfies 
\begin{equation}\label{eq1}
\sup_{s \in \R }| (1 + a^2 s^2) \, (1 - b^{2} \frac{\wrt ^2}{ds^2} ) g(s) | < \infty
\end{equation}
for all $a, b \in \R^+$.
Then 
\[
\int_{-\infty}^{+\infty} | (1 - b^{2} \frac{\wrt ^2}{ds^2} ) g(s) | \, ds \leq \frac{\pi}{a} \sup\{ (1 + a^2 s^2) \, (1 - b^{2} \frac{\wrt ^2}{ds^2} ) g(s) | : s \in \R \} ,
\] 
whence, taking Fourier transforms, we see that
\[
(1 + b^{2} t^2 ) \hat g(t) | \leq \frac{\pi}{a} \sup\{ (1 + a^2 s^2) \, (1 - b^{2} \frac{\wrt ^2}{ds^2} ) g(s) | : s \in \R \} 
\]
for all $t \in \R$, and so 
\[
\int_{-\infty}^{+\infty} | \hat g(t) | \, dt \leq  \frac{ \pi^2 }{ab} \sup\{ (1 + a^2 s^2) \, (1 - b^{2} \frac{\wrt ^2}{ds^2} ) g(s) | : s \in \R \} .
\]
Take $b$ to be $a^{-1}$ and then minimise in $a$; then
\[
\begin{aligned}
&\int_{-\infty}^{+\infty} | \hat g(t) | \, dt \\
&\qquad\leq  \pi^2 \min \{  \| g(s) \|_\infty +  a^2 \| s^2 g(s) \|_\infty +   a^{-2} \| g''(s) \|_\infty +   \| s^2 g''(s) \|_\infty : a \in \R^+ \} \\
&\qquad = \pi^2  ( \| g(s) \|_\infty +  \| s^2 g(s) \|_\infty^{1/2} \| g''(s) \|_\infty^{1/2} +   \| s^2 g''(s) \|_\infty ) .
\end{aligned}
\]  

We apply this estimate to the functions $\tilde h(\lambda)$ and $\lambda \tilde h'(\lambda)$, remembering that Laplace transforms are rotated versions of Fourier transforms. 
\end{proof}
We now apply our previous results to functions $m$ of Laplace transform type. 
In particular, it is easy to check that all the functions $m(\lambda) = \expe ^{-u \lambda^\alpha}$ may be treated using this lemma.

\begin{cor}
Suppose that $0 < \theta < (1 - \alpha)\pi/2$.
Then there exists a constant $C$ such that
\[
\sup \{ |\exp(-uL^\alpha) f| : u \in \Gamma_\theta \} \leq C A^*f .
\]
Hence the operator $f \mapsto \sup \{ |\exp(-uL^\alpha) f| : u \in \Gamma_\theta \}$ is of weak type $(1,1)$.
\end{cor}

Cowling \cite{Cowling} found $L^p$ estimates for this maximal function when $p>1$, but did not prove a weak-type $(1,1)$ estimate.

\section{Characterisations of Hardy spaces by  subordinated semigroups $\exp(-tL^{\alpha})$ for $0 < \alpha < 1 $}

In section, we assume that $(X,d,\mu)$ is a metric space endowed with a metric $d$ and a nonnegative Borel measure $\mu$; our arguments work in quasimetric spaces but this involves extra constants that may obscure the proof. 
We assume that there exists $n\in \R^+$ such that
\begin{equation}\label{eq-volume of the ball}
\mu(B(x,r))\eqsim r^n
\qquad\forall x\in X \quad\forall r\in \R^+, 
\end{equation}
where $B(x,r)=\{y\in X: \varrho(x,y)<r\}$.

Suppose that $L$ is a positive self-adjoint operator on $L^2(X)$. Suppose further that the kernel $p_t$ of $\exp(-tL)$ admits a Poisson upper bound, that is, there exist $m$ and $\delta_0$ in $\R^+$ such that for all $x,y\in X$ and $t\in \R^+$,
\begin{equation}\label{PU} 
|p_t(x,y)|
\lesssim \frac{1}{t^{n/m}}\Big(\frac{t^{1/m}}{t^{1/m}+\varrho(x,y)}\Big)^{n+\delta_0}.
\end{equation}

Denote by  $p_{k,t}(x,y)$ the kernel of $(tL)^k\exp(-tL)$ for $k\in \N^+$. 
Note that conditions \eqref{eq-volume of the ball} and \eqref{PU} imply that for any $\delta\in (0,\delta_0)$ and $k\in \N^+$, 
\begin{equation}\label{qt}
|p_{k,t}(x,y)|\lesssim \frac{1}{t^{n/m}}\Big(\frac{t^{1/m}}{t^{1/m}+\varrho(x,y)}\Big)^{n+\delta}
\end{equation}
for all $t\in \R^+$ and $x,y\in X$. 
See, for example, \cite{CD}. 
Note that if $\exp(-tL)$ admits a Gaussian upper bound, that is, there exist $m\ge 1$ and $ c\in \R^+$ such that for all $x,y\in X$ and $t\in \R^+$,
\begin{equation*}
|p_t(x,y)|\lesssim \frac{1}{t^{n/m}}\exp\Big(-\frac{\varrho(x,y)^{m/(m-1)}}{ c t^{1/(m-1)}}\Big),
\end{equation*}
then the estimate \eqref{PU} holds for all $\delta_0\in \R^+$. 

\begin{defn}\label{defn-Hardy}
The area square function associated to $L^\alpha$ is defined by
\[
S_{L^\alpha}f(x)
:=\biggl(\int_0^\infty\int_{\varrho(x,y)<t}
|t^{m\alpha}L^\alpha \exp(-t^{m\alpha}L^{\alpha})f(y)|^2 \,\frac{\wrt \mu(y) \,\wrt t}{t^{n+1}}\biggr)^{1/2}.
\]

For $p\in ({n}/{(n+\delta_0)},1]$,  the Hardy space  $H^p_{L^\alpha}(X)$ is defined to be the completion of the set
\[
\left\{f\in L^2(X):S_{L^\alpha}f\in L^p(X)\right\}
\]
in the quasi-norm $\|f\|_{H^p_{L^\alpha}}:=\|S_{L^\alpha}f\|_{L^p}$.
\end{defn}

We now verify that $ \| \cdot \|_{H^p_{L^\alpha}} $ in Definition  \ref{defn-Hardy} is indeed a quasi-norm. 
As argued in \cite{HLMMY},  if $f\in L^2(X)$, $p\in (0,1)$, and $\|S_{L^\alpha}f\|_{L^p}=0$, then $tL^\alpha \exp(-tL^{\alpha})f(x)=0$ for almost all $x\in X$ and all $t\in \R^+$.  
Since
\[
\exp(-tL^{\alpha})-I =-\int_0^t L^\alpha \exp(-sL^{\alpha}) \,\wrt s,
\]
$\exp(-tL^{\alpha})f(x) = f(x)$ for almost all $x \in X$ and all $t\in\R^+$. 
On the other hand, from the upper bound on the kernel of $\exp(-tL^{\alpha})$ (see Lemma \ref{lem- heat kernel alpha} below),
\[
\|f\|_{L^\infty(X)}=\|\exp(-tL^{\alpha})f\|_{L^\infty(X)}\lesssim t^{-\frac{n}{2m\alpha}}\|f\|_2
\] 
for all $t\in\R^+$.
Letting $t\to \infty$, we conclude that $f=0$ almost everywhere. 
This confirms that $\|\cdot\|_{H^p_{L^\alpha}} $ in Definition  \ref{defn-Hardy} is a quasi-norm.

It is natural to ask whether the Hardy spaces $H^p_{L^\alpha}(X)$ depend on $\alpha$. 
Our first main result shows that these spaces coincide and have equivalent norms.

\begin{thm}\label{mainthm-Hardy}
Let $\alpha\in (0,1)$ and $p\in ({n}/{(n+\delta_0)},1]$. Then the Hardy spaces $H^p_{L^\alpha}(X)$ and $H^p_{L}(X)$ coincide and have equivalent norms.
\end{thm}
The equivalence of $H^p_{L^\alpha}$ and $H^p_L$ was shown earlier in \cite{HLMMY} for the case where $p=1$ and $\alpha=1/2$, and where $L$ is a positive self-adjoint operator satisfying Davies--Gaffney estimates. 
To the best of our knowledge, our paper is the first to prove the equivalence for general $\alpha\in (0,1)$.

As a consequence of Theorem \ref{mainthm-Hardy} we obtain the equivalence of BMO spaces associated to $L^\alpha$. We first recall the definition of these BMO spaces. 
For all $\beta\in \R^+$, we define the class $\mathcal{M}_\beta$ to be the set of all functions  $f\in L^2_{\loc}(X)$ such that 
\begin{equation}\label{eq-Mbeta}
\|f\|_{\mathcal{M}_{\beta}}:=\Big(\int_{X}\frac{|f(x)|^2}{(1+|x|)^{n+\beta}} \,\wrt \mu(x)\Big)<\infty.
\end{equation}
For $\alpha\in (0,1]$, we set $\mathcal{M}_{L^\alpha}=\bigcup_{0<\beta<\delta_0}\mathcal{M}_\beta$.

\begin{defn}\label{defn-BMO}
Let $\alpha\in (0,1]$ and $0\le \nu< {\delta_0}/{n}$. 
We say that $f\in \mathcal{M}_{L^\alpha}$ is in $\BMO^\nu_{L^\alpha}$ if 
\begin{equation}
\|f\|_{\BMO^\nu_{L^\alpha}}
:=\sup_B\biggl(\frac{1}{|B|^{1+2\nu}}\int_B|(I-\exp( r_B^{m\alpha}L^{\alpha}))f(x)|^2 \,\wrt \mu(x) \biggr)^{1/2} ,
\end{equation}
where the supremum is taken over all balls $B$, and $r_B$ is the radius of the ball $B$.
\end{defn}

The following result is a consequence of Theorem \ref{mainthm-Hardy}.

\begin{thm}\label{thm-BMO}
Let $\alpha\in (0,1)$ and let $0\le \nu< {\delta_0}/{n}$. 
Then 
\[
\BMO^\nu_{L^\alpha}\equiv \BMO^\nu_{L}.
\]
\end{thm}

Our next main result is a characterisation of the Hardy spaces $H^p_{L^\alpha}$ in terms of the Littlewood--Paley square function. 
We define the square function associated to $L^\alpha$ by
\[
G_{L^\alpha}f
:= \biggl(\int_0^\infty|t^{m\alpha}L^\alpha \exp(-t^{m\alpha}L^{\alpha}) f|^2 \,\frac{\wrt t}{t}\biggr)^{1/2}.
\]

Our second result is the following theorem.

\begin{thm}\label{mainthm-Hardy Charac}
Let $\alpha\in (0,1]$ and $p\in ({n}/{(n+\delta_0)},1]$, and take $f\in L^2(X)$. 
Then $f\in H^p_{L^\alpha}(X)$ if only if $G_{L^\alpha}f\in L^p(X)$; moreover, 
\begin{equation}\label{eq1-mainthm}
\|f\|_{H^p_{L^\alpha}(X)}\eqsim \|G_{L^\alpha}f\|_{L^p(X)}.
\end{equation}
\end{thm}

From Theorem \ref{mainthm-Hardy} and Theorem \ref{mainthm-Hardy Charac}, the following result follows immediately.
\begin{cor}\label{cor}
Let $\alpha\in (0,1]$ and $p\in ({n}/{(n+\delta_0)},1]$, and take $f\in L^2(X)$. 
Then $f\in H^p_L(X)$ if only if $G_{L^\alpha}f\in L^p(X)$; moreover,
\begin{equation}\label{eq-mainthm}
\|f\|_{H^p_L(X)}\eqsim \|G_{L^\alpha}f\|_{L^p(X)}.
\end{equation}
\end{cor}
The problem of characterising the Hardy spaces associated to an operator by a Littlewood--Paley square function has attracted some attention. 
In \cite{DJL}, the authors assume that $L$ is a positive self-adjoint operator, has a Gaussian upper bound and satisfies a Moser-type inequality, and prove \eqref{eq-mainthm}  for the particular cases where  $\alpha=1$ and $\alpha=1/2$. 
Recently, under the weaker conditions that $L$ is a positive self-adjoint operator and has a Gaussian upper bound, the equivalence \eqref{eq-mainthm} was obtained in \cite{Hu} for the case $\alpha=1$.  
Our results in Theorem \ref{mainthm-Hardy Charac} and Corollary \ref{cor} are significant improvements of those in \cite{DJL, Hu}, as we are able to prove \eqref{eq-mainthm} for the full range of $\alpha\in (0,1)$  under the weaker condition of the Poisson upper bound \eqref{PU}.

Recall that an argument using Laplace transforms and contour integration shows that
\begin{equation}\label{eq-sub formula}
\expe ^{-t s^\alpha} = \int_0^\infty p_t(u) \, \expe ^{-us} \,\wrt u
\qquad\forall s, t  \in \R^+
\end{equation}
if and only if
\[
\begin{aligned}
p_t(u)
&= \frac{1}{2\pi i} \int_{0-i\infty}^{0+i\infty} \expe ^{-t z^\alpha + uz} \,\wrt z  \\
&= \frac{1}{2\pi } \int_{-\infty}^{+\infty} \expe ^{-t (iy)^\alpha + iuy} \,\wrt y  \\
&= \mathcal{F}^{-1} m_t(u) ,
\end{aligned}
\]
where $\mathcal{F}$ denotes the Fourier transformation and
\[
m_t(y) = \expe ^{-t (iy)^\alpha} = \expe ^{-t[\cos(\pi\alpha/2) + i \sin(\pi\alpha/2){\rm sgn}(y)] \, |y|^\alpha}
\qquad\forall y \in  \R.
\]
We now recall some basic properties of $p_t(u)$ (see the discussion leading up to \eqref{eq:p-estimates} or \cite{Gri}).
\begin{lem}\label{lem-ptu} 
Let $\alpha\in (0,1)$. Then 
\begin{enumerate}[(i)]
\item $p_t(u) = t^{-1/\alpha} \, p_1(t^{-1/\alpha}u )$ for all $t, u>0$;
\item $p_1(u)\lesssim_N u^N$ for all $N\in \R^+$ when $u<1$;
\item $p_1(u) \lesssim u^{-1-\alpha}$ when $u\ge 1$.
\end{enumerate}
\end{lem}
Let $\alpha\in (0,1)$, and denote by $p^\alpha_{t}$ and $q^\alpha_t$ the kernels of $\exp(-tL^{\alpha})$ and $tL^\alpha \exp(-tL^{\alpha})$. 

\begin{lem}\label{lem- heat kernel alpha}
For all $\delta\in (0,\delta_0)$,
\begin{equation}\label{eq-kernel palpha}
|q^\alpha_t(x,y)|+|p^\alpha_t(x,y)|
\lesssim_\delta \frac{1}{t^{{n}/{m\alpha}}}\left(\frac{t^{{1}/{m\alpha}}}{t^{{1}/{m\alpha}}+\varrho(x,y)}\right)^{n+\delta}
\end{equation}
for all $t\in \R^+$ and all $x,y\in X$.
\end{lem}
\begin{proof}
It follows from \cite[Lemma 5.4]{Gri} that for all $\delta\in (0,\delta_0)$, 
\begin{equation}\label{eq1-proof kernel palpha}
|p^\alpha_t(x,y)|
\lesssim \frac{1}{t^{{n}/{m\alpha}}}\left(\frac{t^{{1}/{m\alpha}}}{t^{{1}/{m\alpha}}+\varrho(x,y)}\right)^{n+\delta} .
\end{equation}
Coupled with \cite[Lemma 2.5]{CD}, this implies a similar estimate for $q^\alpha_t(x,y)$, and  we are done.
\end{proof}

In what follows, we write $a\vee b =\max\{a,b\}$ and $a\wedge b =\min \{a,b\}$. 
Setting $Q(z)=z\expe ^{-z}$ and $Q_t(z)=Q(t^mz)$ for $t\in \R^+$, we denote by $h_{s,t}$ the kernel of $Q_s(L)Q_{t^\alpha}(L^\alpha)$. 

\begin{lem}\label{lem-Qst}
Let $\alpha\in (0,1)$. Then for all $\delta\in (0,\delta_0)$,
\begin{equation}\label{eq-hts}
|h_{s,t}(x,y)|\lesssim \Big(\frac{s\wedge t}{s\vee t}\Big)^{m\alpha}\frac{1}{(s\vee t)^{n}}\Big(\frac{s\vee t}{s\vee t+\varrho(x,y)}\Big)^{n+\delta}
\end{equation}
for all $s,t\in \R^+$ and all  $x,y\in X$.
\end{lem}
\begin{proof}
From \eqref{eq-sub formula}, 
\begin{equation*}
\exp(-t^{m\alpha}L^{\alpha}) = \int_0^\infty t^{-m} \, p_1(t^{-m}u ) \, \exp(-uL) \,\wrt u.
\end{equation*}
We change variables and rewrite:
\begin{equation*}
\exp(-t^{m\alpha}L^{\alpha}) = \int_0^\infty p_1(u ) \, \exp(-ut^{m}L) \,\wrt u.
\end{equation*}
By differentiating with respect to $t$ and then multiplying by $t$, we obtain
\begin{equation*}
-m\alpha t^{m\alpha} L^\alpha \exp(-t^{m\alpha}L^{\alpha}) = -m\int_0^\infty p_1(u ) \, ut^mL\exp(-ut^{m}L) \,\wrt u.
\end{equation*}
This implies that
\begin{equation*}
Q_{t^\alpha}(L^\alpha) = \alpha^{-1}\int_0^\infty p_1(u ) \, ut^mL\exp(-ut^{m}L) \,\wrt u.
\end{equation*}
Hence
\begin{equation*}
Q_s(L)Q_{t^\alpha}(L^\alpha) = \alpha^{-1}\int_0^\infty p_1(u ) \, ut^ms^mL^2\expe ^{-(ut^m+s^m)L} \,\wrt u.
\end{equation*}
From \eqref{qt}, for any $\delta\in (0,\delta_0)$,
\begin{equation}\label{eq3-sub formula}
\begin{aligned}
&|h_{s,t}(x,y)|\\
&\qquad\lesssim \int_0^\infty p_1(u ) \frac{ut^ms^m}{(ut^m+s^m)^2}\frac{1}{(ut^m+s^m)^{n/m}}\Big(\frac{(ut^m+s^m)^{1/m}}{(ut^m+s^m)^{1/m}+\varrho(x,y)}\Big)^{n+\delta} \,\wrt u\\
&\qquad=\int_0^1\ldots +\int_1^\infty\ldots=:I_1+I_2.
\end{aligned}
\end{equation}
First we treat the term $I_1$. 
By using Lemma \ref{lem-ptu} (i), (ii) and the inequalities
\[
\frac{1}{u a+b}\le \frac{1}{u(a+b)} 
\qquad\text{and}\qquad 
ua+b\le a+b 
\]
for all $a,b \in \R^+$ and all $u \in (0,1]$, we see that $I_1$ is dominated by
\[
\int_0^1 u^N \frac{ut^ms^m}{u^2(t^m+s^m)^2}\frac{1}{u^{n/m}(t^m+s^m)^{n/m}}\frac{1}{u^{{(n+\delta)}/{m}}}\Big(\frac{(t^m+s^m)^{1/m}}{(t^m+s^m)^{1/m}+\varrho(x,y)}\Big)^{n+\delta} \,\wrt u ,
\]
where $N\in \R^+$ in Lemma \ref{lem-ptu}  (ii) is chosen so that $N-1- \frac{2n+\delta}{m} \ge 0$. 
This implies that 
\begin{equation}\label{eq-I1}
\begin{aligned}
I_1&\lesssim \frac{t^ms^m}{(t^m+s^m)^2}\frac{1}{(t^m+s^m)^{n/m}}\Big(\frac{(t^m+s^m)^{1/m}}{(t^m+s^m)^{1/m}+\varrho(x,y)}\Big)^{n+\delta}\\
&\eqsim \Big(\frac{s\wedge t}{s\vee t}\Big)^m\frac{1}{(s\vee t)^{n}}\Big(\frac{s\vee t}{s\vee t+\varrho(x,y)}\Big)^{n+\delta}.
\end{aligned}
\end{equation}

We next take care of the term $I_2$. By Lemma \ref{lem-ptu} (iii),
\[
I_2\lesssim \int_1^\infty \frac{1}{u^\alpha} \frac{ut^ms^m}{(ut^m+s^m)^{2}}\frac{(ut^m+s^m)^{\delta/m}}{[(ut^m+s^m)^{1/m}+\varrho(x,y)]^{n+\delta}} \,\frac{\wrt u}{u^{1/2} } \,.
\]
We now consider two cases: $s\le t$ and $s>t$.

If $s\le t$, then $ut^m+s^m\eqsim ut^m\ge t^m$ as $u\ge 1$. Hence
\begin{equation}\label{eq-I2 case 1}
\begin{aligned}
I_2&\lesssim \int_1^\infty \frac{1}{u^\alpha} \frac{ut^ms^m}{u^2t^{2m}}\frac{u^{\delta/m}t^\delta}{(t+\varrho(x,y))^{n+\delta}} \,\frac{\wrt u}{u}\\
&\eqsim \Big(\frac{s}{t}\Big)^m\frac{t^\delta}{(t+\varrho(x,y))^{n+\delta}}.
\end{aligned}
\end{equation} 

If $s>t$, then we split $I_2$ into two parts:
\[
I_2= \int_1^{(s/t)^m}\ldots + \int^\infty_{(s/t)^m}\ldots =: I_{21}+I_{22}.
\]
Note that $ut^m +s^m\eqsim s^m$ as $u\le (s/t)^m$. Therefore
\begin{equation}\label{eq-I21 case 2}
\begin{aligned}
I_{21}  &\eqsim \int_1^{(s/t)^m} \frac{1}{u^\alpha} \frac{ut^ms^m}{s^{2m}}\frac{s^{\delta}}{(s+\varrho(x,y))^{n+\delta}} \,\frac{\wrt u}{u}\\
&\eqsim \Big(\frac{s}{t}\Big)^{m\alpha}\frac{s^\delta}{(s+\varrho(x,y))^{n+\delta}} \,.
\end{aligned}
\end{equation}
Since $ut^m +s^m\eqsim ut^m\ge s^m$ as $u>(s/t)^m$, 
\begin{equation}\label{eq-I22 case 2}
\begin{aligned}
I_{22}  &\lesssim \int^\infty_{(s/t)^m}\frac{1}{u^\alpha} \frac{ut^ms^m}{u^2t^{2m}}\frac{u^{\delta/m}t^\delta}{(s+\varrho(x,y))^{n+\delta}} \,\frac{\wrt u}{u}\\
&\eqsim \Big(\frac{t}{s}\Big)^{m\alpha}\frac{s^\delta}{(s+\varrho(x,y))^{n+\delta}}.
\end{aligned}
\end{equation}  
Taking \eqref{eq3-sub formula}, \eqref{eq-I1}, \eqref{eq-I2 case 1}, \eqref{eq-I21 case 2} and \eqref{eq-I22 case 2} into account, we deduce \eqref{eq-hts}, as desired.
\end{proof}

We are ready to prove the main results in this section.
\begin{proof}
[Proof of Theorem \ref{mainthm-Hardy}] 
We first show that $H^p_{L^\alpha}(X)\cap L^2(X)$ injects continuously into $H^p_L(X)$ for all $p\in ({n}/{(n+\delta_0\alpha)},1]$. 
It was proved in \cite[Proposition 3.3]{Y} that $f\in H^p_{L^\alpha}(X)\cap L^2(X)$ admits a molecular decomposition: $f=\sum_{j}\lambda_j a_j$ such that $\sum_{j}|\lambda|^p\eqsim \|f\|_{H^p_{L^\alpha}(X)}$ and each $a_j$ is a $p$-molecule, that is,
\[
a_j(x)
=c\int_0^\infty t^{m\alpha}L^\alpha  \exp(-t^{m\alpha}L^{\alpha}) (I-\exp(t^{m\alpha}L^{\alpha}))(A_j(t,\cdot))(x) \,\frac{\wrt t}{t} \,,
\]
where the function $A_j$ is supported in a tent $\widehat{B}_j:= \{(x,t)\in X\times \R^+: |x-x_{j}|<r_{j}-t\}$ over some ball $B(x_{j}, r_{j})$,
\[
\int_{\widehat{B_j}}|A_j(t,x)|^2 \,\frac{\wrt x\,\wrt t}{t}\le |B|^{1-2/p},
\]
and $c=\bigl( \int_0^\infty z\expe ^{-z}(I-\expe ^{-z})\,{dz}/{z} \bigr)^{-1}$.

Hence it suffices to prove that there exists $C\in \R^+$ such that
\begin{equation}\label{eq- SLa}
\|S_L a_j\|_{L^p(X)}\le C
\end{equation}
for all $p$-molecules $a_j$. 
The proof of \eqref{eq- SLa} is similar to that of the estimate (3.11) in \cite{Y}, the only difference being that we use the kernel estimate of $Q_s(L)Q_{t^\alpha}(L^\alpha)$ in Lemma \ref{lem-Qst} instead of an estimate for the kernel of $Q_s(L)Q_{t}(L)$. 
We omit the details.

We prove similarly that $H^p_{L}(X)\cap L^2(X)$ injects continuously into $H^p_{L^\alpha}(X)$.
\end{proof}

\begin{proof}
[Proof of Theorem \ref{thm-BMO}:] 
From Theorem \ref{mainthm-Hardy}, if $\alpha\in (0,1]$, then
\begin{equation}\label{eq- Hardy L*}
H^p_{L}(X)\equiv H^p_{L^\alpha}(X) 
\qquad\forall p\in \left({n}/{(n+\delta_0}),1\right].
\end{equation}
However, by \cite[Theorem 4.1]{Y}, 
\[
(H^p_{L^\alpha}(X))^*\equiv \BMO^{{1}/{p}-1}_{L^\alpha}(X)
\qquad\forall p\in \left({n}/{(n+\delta_0)},1\right].
\]
Together with \eqref{eq- Hardy L*}, this implies that, for all $\alpha\in (0,1)$,
\[
\BMO^{{1}/{p}-1}_{L}(X)\equiv \BMO^{{1}/{p}-1}_{L^\alpha}(X)
\qquad\forall p\in \left({n}/{(n+\delta_0)},1\right],
\]
or equivalently,
\[
\BMO^{\nu}_{L}(X)\equiv \BMO^{\nu}_{L^\alpha}(X) 
\qquad\forall \nu \in \left[ 0, {\delta_0}/{n} \right),
\]
as desired.
\end{proof}

In order to prove Theorem \ref{mainthm-Hardy Charac}, we need the following elementary estimate. We omit the  proof since it is simple.

\begin{lem}\label{lem-ele est}
Let $\epsilon \in \R^+$. 
Then
\[
\int_{X}\frac{1}{s^n}\Big(1+\frac{\varrho(x,y)}{s}\Big)^{-n-\epsilon}|f(y)| \,\wrt \mu(y)
\lesssim \mathcal{M}f(x).
\]
for all $x\in X$ and $s\in \R^+$, where $\mathcal{M}$ is the Hardy--Littlewood maximal function on $X$.
\end{lem}

\bigskip

Recall that $Q(z)=z\expe ^{-z}$ and $Q_t(z)=Q(t^mz)$ for all $t\in \R^+$. Then
\begin{equation}\label{eq-indentity}
1= c_m\int_0^\infty Q_t(z)Q_t(z) \,\frac{\wrt t}{t}
=c_m\int_0^\infty (t^mz)^{2}\expe ^{-2t^mz} \,\frac{\wrt t}{t}.
\end{equation}
where $c_m =\bigl( \int_0^\infty Q(z)^2\,{dz}/{z}\bigr)^{-1}$.

We now define
\begin{equation}\label{eq1-Phi}
\Phi(z)=c_m\int_1^\infty (s^mz)^{2}\expe ^{-2s^mz} \,\frac{\wrt s}{s}.
\end{equation}
This implies that 
\begin{equation}\label{eq1s-Phi}
\begin{aligned}
\Phi_t(z):=\Phi(t^mz)&=c_m\int_1^\infty ((st)^mz)^{2}\expe ^{-2(st)^mz} \,\frac{\wrt s}{s}\\
&=\frac{c_m}{m} \int_{t^mz}^\infty s \expe ^{-2s} \,\wrt s.
\end{aligned}
\end{equation}
By integration by parts,
\begin{equation}\label{eq2-Phi}
\Phi_t(L)=-\frac{c_m}{2m}t^{m}L \exp(-2t^mL) - \frac {c_m}{4m} \exp(-2t^mL).
\end{equation}    
Together with \eqref{PU} and \eqref{qt}, this implies that, for all $\delta\in (0,\delta_0)$,
\begin{equation}\label{eq-kernel of Phi}
|K_{\Phi_t(L)}(x,y)|\leq \frac{C}{t^{n/m}}\Big(\frac{t^{1/m}}{t^{1/m}+\varrho(x,y)}\Big)^{n+\delta}
\end{equation}
where $K_{\Phi_t(L)}$ is the kernel of $\Phi_t(L)$.

We now define, for $\epsilon\in \R^+$,
\[
Q^*_{L,\epsilon}f(t,x)=\sup_{s\in \R^+}\sup_{y\in X}\frac{|Q_s(L)f(y)|}{(1+\varrho(x,y)/s)^\lambda}\Big(\frac{s}{t}\wedge \frac{t}{s}\Big)^\epsilon.
\]

We are now ready to prove Theorem \ref{mainthm-Hardy Charac}.
\begin{proof}[Proof of Theorem \ref{mainthm-Hardy Charac}]
We suppose that $\alpha=1$; the case where $\alpha\in (0,1)$ is similar.

We first prove that $\|G_Lf\|_{L^p(X)}\lesssim \|S_Lf\|_{L^p(X)}$ for all $f\in H^p_L(X)\cap L^2(X)$. 
As in the proof of Theorem \ref{mainthm-Hardy}, it suffices to prove that 
\begin{equation}\label{eq- GLa}
\|S_L a\|_{L^p(X)}\le C
\end{equation}
for all $p$-molecules $a$. 
The proof of \eqref{eq- GLa} is standard and similar to that of the estimate (3.11) in \cite{Y},
and we omit the details.

It remains to show that $\|S_Lf\|_{L^p(X)}\lesssim \|G_Lf\|_{L^p(X)}$. 
To do this, we choose $\delta\in (0,\delta_0)$, $\theta\in (0,p)$ and $\lambda\in (0,n+\delta)$ such that $p<{n}/{(n+\delta)}$ and  $\lambda\theta>n$.

From \eqref{eq1s-Phi}, for any $s\in \R^+$,
\begin{equation}\label{eq1-prop1}
f= \Phi_s(L)f+ c_m\int_0^{s} Q_u(L)Q_u(L)f \,\frac{\wrt u}{u}.
\end{equation}
For all $j\in \Z$, all  $s, t\in [1,2]$, and all $y\in X$,
\begin{equation}\label{eq1-proof prop1}
\begin{aligned}
Q_{s2^j}(L)f(y)&= \Phi_{2^j}(L)Q_{s2^j}(L)f(y)+ c_m\int_0^{2^j} Q_u(L)Q_u(L)[Q_{s2^j}(L)f](y) \,\frac{\wrt u}{u}
\\
&= A(y)+B(y),
\end{aligned}
\end{equation}
say, from \eqref{eq1-prop1}.
Note that
\[
\begin{aligned}
\Phi_{2^j}(L)Q_{s2^j}(L)=\Big(\frac{2s}{t}\Big)^m \Phi_{2^j}(L)
\exp( -[(2s)^m-t^m]2^{(j-1)m}L) Q_{t2^{j-1}}(L).
\end{aligned}
\]
From \eqref{eq2-Phi}, and the fact that $[(2s)^m-t^m]2^{(j-1)m}\lesssim 2^{jm}$, the kernel $K_{j,s,t}$ of the operator $\Phi_{2^j}(L) \exp(-[(2s)^m-t^m]2^{(j-1)m}L)$ has a Poisson upper bound:
\[
|K_{j,s,t}(x,y)(x,y)|\leq \frac{C}{2^{jn}}\Big(1+
\frac{\varrho(x,y)}{2^j}\Big)^{-n-\delta} .
\] 

Hence, for all $\ell \in \Z $
\begin{equation*}
\begin{aligned}
&\frac{|A(y)|}{(1+\varrho(x,y)/(s2^j))^\lambda}\Big(\frac{s2^j}{t2^\ell}\wedge \frac{t2^\ell}{s2^j}\Big)^\epsilon\\
&\qquad\lesssim \Big(\frac{s2^j}{t2^\ell}\wedge \frac{t2^\ell}{s2^j}\Big)^\epsilon\Big(1+
\frac{\varrho(x,y)}{s2^j}\Big)^{-\lambda}\int_{X} \frac{1}{2^{jn}}\Big(1+
\frac{\varrho(y,z)}{2^j}\Big)^{-n-\delta}|Q_{t2^{j-1}}(L)f(z)| \,\wrt\mu(z).
\end{aligned}
\end{equation*}
Since $0<\lambda<n+\delta$ and $s\in [1,2]$, 
\[
\Big(1+
\frac{\varrho(x,y)}{s2^j}\Big)^{-\lambda}\Big(1+
\frac{\varrho(y,z)}{2^j}\Big)^{-n-\delta}\le \Big(1+
\frac{\varrho(x,z)}{2^j}\Big)^{-\lambda}.
\]
Hence
\begin{equation}\label{eq2a-proof prop1}
\begin{aligned}
&\frac{|A(y)|}{(1+\varrho(x,y)/(s2^j))^\lambda} \Big(\frac{s2^j}{t2^\ell}\wedge \frac{t2^\ell}{s2^j}\Big)^\epsilon\\
&\qquad \lesssim \Big(\frac{s2^j}{t2^\ell}\wedge \frac{t2^\ell}{s2^j}\Big)^\epsilon\int_{X} \frac{1}{2^{jn}}\Big(1+
\frac{\varrho(x,z)}{2^j}\Big)^{-\lambda}|Q_{t2^{j-1}}(L)f(z)| \,\wrt\mu(z).
\end{aligned}
\end{equation}
We note that, since $s,t\in [1,2]$, 
\begin{equation}\label{eq2b-proof prop1}
\begin{aligned}
\frac{|Q_{t2^{j-1}}(L)f(z)|}{(1+\varrho(x,z)/2^j)^\lambda}\Big(\frac{s2^j}{t2^\ell}\wedge \frac{t2^\ell}{s2^j}\Big)^\epsilon
&\eqsim \frac{|Q_{t2^{j-1}}(L)f(z)|}{(1+\varrho(x,z)/(t2^{j-1}))^\lambda}\Big(\frac{t2^j}{t2^\ell}\wedge \frac{t2^\ell}{t2^j}\Big)^\epsilon\\
&\lesssim Q^*_{L,\epsilon}f(t2^\ell,x).
\end{aligned}
\end{equation}
Together with \eqref{eq2a-proof prop1} and our choices that $\theta\in (0,p)$ and $\lambda\theta > n$, this implies that
\begin{equation}\label{eq2-proof prop1}
\begin{aligned}
&\frac{|A(y)|}{(1+\varrho(x,y)/(s2^j))^\lambda} \Big(\frac{s2^j}{t2^\ell}\wedge \frac{t2^\ell}{s2^j}\Big)^\epsilon\\
&\qquad \lesssim [Q^*_{L,\epsilon}f(t2^\ell,x)]^{1-\theta} \Big(\frac{s2^j}{t2^\ell}\wedge \frac{t2^\ell}{s2^j}\Big)^{\theta\epsilon}\\
&\qquad\qquad\times \int_{X} \frac{1}{2^{jn}}\Big(1+
\frac{\varrho(x,z)}{2^j}\Big)^{-\lambda\theta}|Q_{t2^{j-1}}(L)f(z)|^\theta \,\wrt \mu(z)\\
&\qquad \lesssim 2^{-\theta\epsilon|j-\ell|}[Q^*_{L,\epsilon}f(t2^\ell,x)]^{1-\theta} \\
&\qquad\qquad\times \int_{X} \frac{1}{2^{jn}}\Big(1+
\frac{\varrho(x,z)}{2^j}\Big)^{-\lambda\theta}|Q_{t2^{j-1}}(L)f(z)|^\theta \,\wrt \mu(z).
\end{aligned}
\end{equation}

For the term $B(y)$,
\begin{equation*}
\begin{aligned}
Q_u(L)Q_u(L)Q_{s2^j}(L)=
\Big(\frac{2s}{t}\Big)^m(u^mL)^2 \exp(-[2u^m+[(4s)^m-t^m]2^{(j-2)m}]L) Q_{t2^{j-2}}(L).
\end{aligned}
\end{equation*}
Now $2u^m+[(4s)^m-t^m]2^{(j-2)m}\eqsim 2^{jm}$ for all $s,t\in [1,2]$ and all $0<u<2^j$. 
Together with \eqref{qt}, this implies that the kernel $H_{j,s,t}$ of $(u^mL)^2 \exp(-[2u^m+[(4s)^m-t^m]2^{(j-2)m}]L)$ satisfies
\[
|H_{j,s,t}(x,y)|\lesssim \Big(\frac{u}{2^j}\Big)^{2m}\frac{C}{2^{jn}}\Big(1+
\frac{\varrho(x,y)}{2^j}\Big)^{-n-\delta}.
\]
Therefore, for all $\epsilon\in \R^+$,
\begin{equation}\label{eq3a-proof prop1}
\begin{aligned}
&\frac{|A(y)|}{(1+\varrho(x,y)/(s2^j))^\lambda}\Big(\frac{s2^j}{t2^\ell}\wedge \frac{t2^\ell}{s2^j}\Big)^\epsilon\\
&\qquad\lesssim \Big(\frac{s2^j}{t2^\ell}\wedge \frac{t2^\ell}{s2^j}\Big)^\epsilon\Big(1+
\frac{\varrho(x,y)}{s2^j}\Big)^{-\lambda} \\
&\qquad\qquad\qquad \times \int_0^{2^j}\int_{X} \Big(\frac{u}{2^j}\Big)^{2m}\frac{1}{2^{jn}}\Big(1+
\frac{\varrho(y,z)}{2^j}\Big)^{-n-\delta}|Q_{t2^{j-2}}(L)f(z)| \,\wrt\mu(z)\frac{\wrt u}{u}.
\end{aligned}
\end{equation}
Similarly to \eqref{eq2b-proof prop1}, 
\[
\frac{|Q_{t2^{j-2}}(L)f(z)|}{(1+\varrho(x,z)/2^j)^\lambda}\Big(\frac{s2^j}{t2^\ell}\wedge \frac{t2^\ell}{s2^j}\Big)^\epsilon\lesssim Q^*_{L,\epsilon}f(t2^\ell,x).
\]
This and \eqref{eq3a-proof prop1} yield that
\begin{equation}\label{eq3-proof prop1}
\begin{aligned}
&\frac{|A(y)|}{(1+\varrho(x,y)/(s2^j))^\lambda}\Big(\frac{s2^j}{t2^\ell}\wedge \frac{t2^\ell}{s2^j}\Big)^\epsilon\\
&\qquad\lesssim 2^{-\theta\epsilon|j-\ell|}[Q^*_{L,\epsilon}f(t2^\ell,x)]^{1-\theta}\int_{X} \frac{1}{2^{jn}}\Big(1+
\frac{\varrho(x,z)}{2^j}\Big)^{-\lambda\theta}|Q_{t2^{j-2}}(L)f(z)|^\theta \,\wrt\mu(z).
\end{aligned}
\end{equation}
From \eqref{eq1-proof prop1}, \eqref{eq2-proof prop1} and \eqref{eq3-proof prop1} we find that, for all $\ell, j\in \Z $ and $s, t\in [1,2]$,   
\begin{equation*}
\begin{aligned}
&\frac{|Q_{s2^j}(L)f(y)|}{(1+\varrho(x,y)/(s2^j))^\lambda} \Big(\frac{s2^j}{t2^\ell}\wedge \frac{t2^\ell}{s2^j}\Big)^\epsilon\\
&\qquad\lesssim \sum_{k\in \Z }2^{-|\ell-k|\theta\epsilon}[Q^*_{L,\epsilon}f(t2^\ell,x)]^{1-\theta}\int_{X} \frac{1}{2^{kn}}\Big(1+
\frac{\varrho(x,z)}{2^k}\Big)^{-\lambda\theta}|Q_{t2^{k-1}}(L)f(z)|^\theta \,\wrt\mu(z)\\
&\qquad\qquad + \sum_{k\in \Z }2^{-|\ell-k|\theta\epsilon}[Q^*_{L,\epsilon}f(t2^\ell,x)]^{1-\theta}\int_{X} \frac{1}{2^{kn}}\Big(1+
\frac{\varrho(x,z)}{2^k}\Big)^{-\lambda\theta}|Q_{t2^{k-2}}(L)f(z)|^\theta \,\wrt\mu(z)\\
&\qquad\lesssim \sum_{k\in \Z }2^{-|\ell-k|\theta\epsilon}[Q^*_{L,\epsilon}f(t2^\ell,x)]^{1-\theta}\int_{X} \frac{1}{2^{kn}}\Big(1+
\frac{\varrho(x,z)}{2^k}\Big)^{-\lambda\theta}|Q_{t2^{k}}(L)f(z)|^\theta \,\wrt\mu(z).
\end{aligned}
\end{equation*}
Taking the supremum over all $j\in \Z $ and $s\in [1,2]$ we obtain
\begin{equation*}
\begin{aligned}
Q^*_{L,\epsilon}f(t2^\ell,x)^{\theta}
&\lesssim \sum_{k\in \Z }2^{-|\ell-k|\theta\epsilon}\int_{X} \frac{1}{2^{kn}}\Big(1+
\frac{\varrho(x,z)}{2^k}\Big)^{-\lambda\theta}|Q_{t2^{k}}(L)f(z)|^\theta \,\wrt\mu(z).
\end{aligned}
\end{equation*}
Consequently, by Minkowski's inequality,
\begin{equation*}
\begin{aligned}
&\biggl[\int_{1}^2Q^*_{L,\epsilon}f(t2^\ell,x))^2\,\frac{\wrt t}{t}\biggr]^{\theta/2} \\
&\qquad\lesssim \sum_{k\in \Z }2^{-|\ell-k|\theta\epsilon}\int_{X} \frac{1}{2^{kn}}\Big(1+
\frac{\varrho(x,z)}{2^k}\Big)^{-\lambda\theta}\biggl[\int_{1}^2|Q_{t2^{k}}(L)f(z)|^2\,\frac{\wrt t}{t} \biggr]^{\theta/2}\,\wrt\mu(z).
\end{aligned}
\end{equation*}
A change of variable implies that
\begin{equation*}
\begin{aligned}
&
\biggl[\int_{2^\ell}^{2^{\ell+1}}Q^*_{L,\epsilon}f(t,x))^2\,\,\frac{\wrt t}{t}\biggr]^{\theta/2} \\
&\qquad\lesssim \sum_{k\in \Z }2^{-|\ell-k|\theta\epsilon}\int_{X} \frac{1}{2^{kn}}\Big(1+
\frac{\varrho(x,z)}{2^k}\Big)^{-\lambda\theta}\biggl[\int_{2^{k}}^{2^{k+1}}|Q_{t}(L)f(z)|^2\,\frac{\wrt t}{t} \biggr]^{\theta/2}\,\wrt\mu(z).
\end{aligned}
\end{equation*}

We now apply Lemma \ref{lem-ele est} and see that
\begin{equation*}
\begin{aligned}
\biggl[\int_{2^\ell}^{2^{\ell+1}}Q^*_{L,\epsilon}f(t,x))^2\,\frac{\wrt t}{t}\biggr]^{\theta/2}
&\lesssim \sum_{k\in \Z }2^{-|\ell-k|\theta\epsilon}\mathcal{M}\biggl(\biggl[\int_{2^{k}}^{2^{k+1}}|Q_{t}(L)f(z)|^2 \,\frac{\wrt t}{t}\biggr]^{\theta/2}\biggr)(x),
\end{aligned}
\end{equation*}
so
\begin{equation*}
\begin{aligned}
\biggl[\int_{2^\ell}^{2^{\ell+1}}Q^*_{L,\epsilon}f(t,x))^2\,\frac{\wrt t}{t}\biggr]^{1/2}
&\lesssim \sum_{k\in \Z }2^{-|\ell-k|\epsilon}\mathcal{M}_\theta\biggl(\biggl[\int_{2^{k}}^{2^{k+1}}|Q_{t}(L)f(z)|^2 \,\frac{\wrt t}{t}\biggr]^{1/2}\biggr)(x) ,
\end{aligned}
\end{equation*}
where $\mathcal{M}_\theta f=[\mathcal{M}(|f|^\theta)]^{1/\theta}$.
Therefore
\[
\sum_{\ell\in \Z }\int_{2^\ell}^{2^{\ell+1}}Q^*_{L,\epsilon}f(t,x))^2\,\frac{\wrt t}{t}\lesssim \sum_{\ell\in \Z }\biggl[\sum_{k\in \Z }2^{-|\ell-k|\epsilon}\mathcal{M}_\theta\biggl(\biggl[\int_{2^{k}}^{2^{k+1}}|Q_{t}(L)f(z)|^2 \,\frac{\wrt t}{t}\biggr]^{1/2}\biggr)(x)\biggr]^2.
\]
From Young's inequality, it follows that
\[
\biggl[\int_{2^\ell}^{2^{\ell+1}}Q^*_{L,\epsilon}f(t,x))^2\,\frac{\wrt t}{t}\biggr]^{1/2}\lesssim \biggl[\sum_{k\in \Z }\mathcal{M}_\theta\Big(\biggl[\int_{2^{k}}^{2^{k+1}}|Q_{t}(L)f(z)|^2 \,\frac{\wrt t}{t}\biggr]^{1/2}\Big)^2(x)\biggr]^{1/2}.
\]
From the Fefferman--Stein maximal inequality, we deduce that
\begin{equation}\label{eq-Q*Q}
\bigg\|\biggl[\int_{0}^{\infty}Q^*_{L,\epsilon}f(t,x))^2\,\frac{\wrt t}{t}\biggr]^{1/2}\biggr\|_{L^p(X)}
\lesssim \biggl\|\biggl[\int_{0}^{\infty}|Q_t(L)f|^2\,\frac{\wrt t}{t}\biggr]^{1/2}\biggr\|_{L^p(X)}.
\end{equation}
On the other hand, it is obvious that
\[
S_Lf(x)\lesssim \biggl[\int_{0}^{\infty}Q^*_{L,\epsilon}f(t,x))^2\,\frac{\wrt t}{t}\biggr]^{1/2} .
\]
This, together with \eqref{eq-Q*Q}, yields the inequality
\[
\|S_Lf\|_{L^p(X)}\lesssim \|G_Lf\|_{L^p(X)},
\]
which completes the proof.
\end{proof}

\bigskip
{\bf Acknowledgements.} 
The authors were supported by grant DP220100285 of the Australian Research Council.
The authors thank Ben Goldys and John Nolan for helpful references on the L\'evy stable distributions, and 
X. T.  Duong thanks El Maati Ouhabaz and Thierry Coulhon for helpful discussions.

\end{document}